\documentclass[a4paper,11pt]{article}
\usepackage[plainpages=false]{hyperref}
\usepackage{amsfonts,latexsym,rawfonts,amsmath,amssymb,amsthm,mathrsfs}
\usepackage{amsmath,amssymb,amsfonts,latexsym,lscape,rawfonts}

\usepackage[all]{xy}
\usepackage{eufrak}
\usepackage{makeidx}         % allows index generation
\usepackage{graphicx,psfrag}

\usepackage{array,tabularx}

\usepackage{setspace}

\newtheorem{thm}{Theorem}[section]

\newtheorem{lem}[thm]{Lemma}
\newtheorem{clm}[thm]{Claim}

\theoremstyle{remark}
\newtheorem{rmk}[thm]{Remark}

\theoremstyle{definition}
\newtheorem{Def}[thm]{Definition}                                        %

\title{Smooth approximations of the  Conical K\"ahler-
Ricci flows}
\author{Yuanqi Wang}
\begin{document}

\maketitle{}
\begin{abstract}
In this note, we show that the conical K\"ahler-Ricci flows introduced in 
\cite{CYW} exist for all time $t\in [0,\infty)$ in the  weak  sense as in Definition \ref{Definition of Weak flow}. As a key ingredient of the proof, we show that  a conical K\"ahler-Ricci flow is actually the limit of a sequence of smooth K\"ahler-Ricci flows. 
\end{abstract}
\tableofcontents
\section{Introduction}

This is a following up note of \cite{CYW}. Let $M$ be a polarized K\"ahler manifold and $D$ is a smooth divisor of the anti-canonical line bundle.  Suppose the  ``twisted" first Chern class
($\beta \in (0,1]$ )\[
C_{1,\beta}  =     C_1(M) - 2\pi (1-\beta) [D]
\] 
has a definite sign. Suppose $\omega_{0}$ is a smooth K\"ahler metric 
in $C_{1,\beta}$ (if $C_{1,\beta}>0$) or $-C_{1,\beta}$ (if $C_{1,\beta}<0$). One important question is to study the existence of the conical K\"ahler-Einstein metric in $(M, [\omega_0], (1-\beta)[D])$. A metric $\omega_\phi$ (cohomologous to $\omega_{0}$) is said to be K\"ahler-Einstein if
\[
Ric(\omega_\phi) =  \beta \omega_\phi + 2\pi (1-\beta)[D].
\]
This problem has been studied carefully by many authors, for instance, \cite{Berman}, \cite{Brendle}, \cite{JMR}, \cite{SongWang}, \cite{LiSun} , \cite{CGP}, \cite{EGZ}  etc.
In particular, ``conical K\"ahler-Einstein metric" is a key ingredient in the recent solution of existence problem for K\"ahler-Einstein metric with positive scalar curvature \cite{CDS0},
\cite{CDS1}, \cite{CDS2}, \cite{CDS3}. In light of these exciting development, we introduce the notion of   conical K\"ahler-Ricci flow  \cite{CYW}
\begin{equation}\label{Definition of CKRF}
{{\partial \omega_g}\over {\partial t}}  = \beta\omega_g -   Ric(g)+ 2\pi(1-\beta)[D],
\end{equation}
 to attack the existence problem of
conical K\"ahler-Einstein metrics and conical K\"ahler-Ricci solitons. With respect to the potential $\phi$, equation (\ref{Definition of CKRF}) is written as 
\begin{equation}\label{equ CKRF potential equation} \frac{\partial \phi}{\partial t}=\log\frac{\omega_{\phi}^n}{\omega_0^n}+\beta\phi+h+(1-\beta)\log|S|^2,\end{equation}
where $h$ is a smooth function.

 In \cite{CYW}, we establish short time existence for (\ref{Definition of CKRF}) and (\ref{equ CKRF potential equation}), initiated from any $(\alpha,\beta)$-conical K\"ahler metric (see Definition \ref{definition of (alpha,beta)-conical metric:single divisor.} for the definition of  $(\alpha,\beta)$-metrics). 
In this  paper  we want to establish the long time existence of this flow in a  weaker sense.

   Most of the notations in this article follows those of \cite{CYW}. For the readers' convenience, we introduce some key definitions from \cite{CYW} here.

\begin{Def} \label{definition of (alpha,beta)-conical metric:single divisor.} $(\alpha, \beta)$ conical K\"ahler metric:\  For any $\alpha \in (0, \min\{{1\over \beta}-1,1\}),\;$ a   K\"ahler form $\omega$ is said to be an $(\alpha, \beta)$  conical K\"ahler metric on $(M, (1-\beta) D)$  if it satisfies the following conditions.
\begin{enumerate}
\item $\omega$ is a closed positive $(1,1)$-current over $M$.
\item For any point $p\in D$, there exists a  holomorphic chart $\{z,u_i,\ i=1,..,n-1\}$
such that in this chart, $\omega$ is quasi-isometric to the standard cone metric
\[\omega_{\beta}=
 \beta^2|z|^{2\beta - 2} \frac{\sqrt{-1}}{2}d z \wedge d \bar z + {\sqrt{-1}\over 2}  \displaystyle \sum_{j=1}^{n-1} d u_j \wedge d \bar u_j.
\]
\item There is a $\phi\in C^{2,\alpha,\beta}(M)$  such that \[\omega=\omega_0+ i \partial \bar \partial \phi.\]   The
definition of the function space $C^{2,\alpha,\beta}(M)$ is due to Donaldson in \cite{Don}. One could also see section 2 of \cite{CYW}.
\end{enumerate}
\end{Def}

 The following model metric defined in \cite{Don} satisfies the above definition.
 \begin{equation*}\label{Model conical metric defined by Donaldson}\omega_{D}=\omega+\delta i \partial \bar \partial |S|^{2\beta}
 ,\ \textrm{where}\ \delta \ \textrm{is a small enough number}.\end{equation*}
 
 Next we define the notion of weak flows, following the definitions of weak conical K\"ahler-Einstein metrics by Gunancia-Paun \cite{GP} and Yao \cite{Yao}.

\begin{Def}\label{Definition of Weak flow}
A solution $\phi(t)$ for  $t\in [0,T)$ with $\phi(0)=\phi_{0}$ to (\ref{Definition of CKRF}) is called a weak conical K\"ahler-Ricci flow if for all $\widetilde{T}<T$, the following hold.  
\begin{itemize}
\item $\phi(t)\in C^{2+\alpha, 1+\frac{\alpha}{2}}[(M\setminus D)\times (0,\widetilde{T})]\cap  C^{\alpha, \frac{\alpha}{2}}[M\times (0,\widetilde{T})]$, $\phi_{0}\in C^{\alpha}(M)$.
\item $\lim_{t\rightarrow 0} |\phi-\phi_0|_{\alpha}=0$ over $M$. 
\item 
$|\frac{\partial \phi}{\partial t}|_{0,[0,\widetilde{T}]}+|tr_{\omega_{D}}\omega_{\phi}|_{0,[0,\widetilde{T}]}+|tr_{\omega_{\phi}}\omega_{D}|_{0,[0,\widetilde{T}]}\leq C(\omega_{\phi_0}, \widetilde{T})$
over $M\setminus D$. 
\item $\omega_{\phi}\geq  C(\widetilde{T})\omega_{D}$ over $M\setminus D$.
\end{itemize}
In particular, a closed positive $(1,1)$-currrent  $\omega=\omega_{D}+i\partial\bar{\partial} \phi$ is called a weak $(\alpha,\beta)$-metric if the following holds.
\begin{itemize}
\item  $|\phi|_{\alpha}\leq C$ over $M$, $\phi$ is $C^{2,\alpha}$ away from $D$.
\item  $\omega_{\phi}\geq  C(\widetilde{T})\omega_{D}$ over $M\setminus D$.
\item
$|tr_{\omega_{D}}\omega_{\phi}|_{0}+|tr_{\omega_{\phi}}\omega_{D}|_{0}\leq C$
over $M\setminus D$. 
\end{itemize}
\end{Def}

\begin{rmk}For convention on H\"older norms, see Definition \ref{Def convention on Holder norms}. From now on whenever we say "conical" or "$(\alpha,\beta)$"  we mean strong conical or strong $(\alpha,\beta)$, to differ from the weak conical cases.  Notice a random weak $(\alpha,\beta)$-metric  does not aprorily possess the correct cone structure along the singular divisor. 
\end{rmk}
Now we state our main theorem on the long-time existence of the weak flow.
\begin{thm}\label{long time existence of the weak flow}Suppose $\omega_{\phi_0}$ is a  $(\alpha,\beta)$-conical metric. 
Then there exists a weak conical K\"ahler-Ricci flow $\omega_{\phi}(t),\ t \in [0,\infty)$ initiated from $\phi_0$. Moreover, the weak flow above  coincides with the  strong flow given by Theorem 1.2 in \cite{CYW} till whenever the strong flow exists.
\end{thm}

\begin{rmk} For smooth K\"ahler-Ricci flows, the global existence is proved by Cao  \cite{Cao}.  For conical K\"ahler-Ricci flow,  when $n=1$, this is recently proved
by Yin \cite{Yin2}, and also by  Mazzeo-Rubinstein-Sesum \cite{MRS} with different function spaces. For higher dimensions,  we believe  the parabolic  version of Brendle's work can solve the long time existence problem  when $\beta\leq \frac{1}{2}$.
\end{rmk}
\begin{rmk} Recently, the author learned  Liu  and Zhang also  consider the conical K\"ahler-Ricci flows on Fano manifolds. In  \cite{LiuZhang},  they also  obtain  long time solutions for conical K\"ahler-Ricci flow on Fano
manifolds. Moreover, they obtain convergence  of their conical
K\"ahler-Ricci flows in some cases. We also have the work by Bahuaud-Vertman on singular Yamabe-flows \cite{Vertman}.
\end{rmk}
%SetThis certainly brings some  room for rose imaginations. Note that the smooth K\"ahler-Ricci flow has been studied for a long time,

\begin{rmk} For simplicity, we only present the case with one smooth divisor. Our proof certainly works with multiple smooth divisors with no intersections and with possibly different angles along each component. 
\end{rmk}

\begin{rmk} If the manifold is not Fano  or  the twisted first Chern has mixed sign,  Theorem  \ref{long time existence of the weak flow} still holds as long as the
 K\"ahler class remains to be fixed along the flow. 
\end{rmk}

The next theorem is almost equivalent to Theorem \ref{long time existence of the weak flow}.

\begin{thm}\label{Existence in the weak flow0}Suppose $\omega_{\phi_0}$ is a weak $(\alpha,\beta)$-metric such that 
\[\omega_{\phi_0}\in C_{1,\beta}(M),\ F_{\phi_0}=\log \frac{|S|^{2-2\beta}\omega^n_{\phi_0}}{\omega^n_{0}}\in C^{2,\alpha,\beta}.\] 
Then there exists a weak conical K\"ahler-Ricci flow $\omega_{\phi}(t),\ t \in [0,\infty)$ initiated from $\phi_0$.

Moreover, if in addition $\phi_0\in C^{2,\alpha,\beta}$, then  the weak flow above is strong and coincides with the strong flow given by Theorem 1.2 in \cite{CYW} till whenever the strong flow exists.
\end{thm}
\begin{rmk}Though Theorem \ref{Existence in the weak flow0} does not require the initial metric to be (strongly) $(\alpha,\beta)$, it needs the volume form of the initial metric to have $C^{2,\alpha,\beta}$ 
regularity, which is a relatively strong condition. Nevertheless, Theorem \ref{Bounding Ricci curvature when t>0} guarantees that the conical flow (\ref{Definition of CKRF}) "smoothes" a
$(\alpha,\beta)$-metric immediately to possess maximal regularity, so the requirements of Theorem \ref{Existence in the weak flow0} are satisfied. 
\end{rmk}

Next we state our result on the smooth approximations of the conical flows.
\begin{thm}\label{Thm Smooth approximation of the conical flow} Suppose the conical flow $\omega_{\phi(t)}$ (solution to (\ref{Definition of CKRF})) exists for $t\in [0,T)$. Then, for any $0<a_0<T$, the smooth flows $\omega_{\phi_{\epsilon}(t)}, t\in [\frac{a_0}{2},T)$ in (\ref{Perturbed CKRF}) approximate 
$\omega_{\phi(t)}$ in $C^{\alpha,\frac{\alpha}{2}}$-sense away from $D$ over $[a_0,T)$ i.e 
\begin{equation*}
\textrm{For all}\ \delta,\ \lim_{\epsilon\rightarrow 0}|\omega_{\phi_{\epsilon}(t)}-\omega_{\phi(t)}|_{C^{\alpha,\frac{\alpha}{2}}, [M\setminus T_{\delta}(D)]\times [a_0,T)}=0.
\end{equation*}
Consequently, for any $t\in [a_0,T)$, $\omega_{\phi_{\epsilon}(t)}$ approximates 
 $\omega_{\phi(t)}$ in the Gromov-Hausdorff sense. 
\end{thm}

\begin{rmk}\label{rmk turbular neighborhood} The $T_{\delta}(D)$ is the turbular neighborhood of the divisor $D$ of radius $\delta$ (with respect to the smooth reference metric $\omega_{0}$). Actually it doesn't matter whether we use $\omega_{0}$ or the conic metric $\omega_{D}$. 
\end{rmk}

\begin{rmk}Theorem \ref{Thm Smooth approximation of the conical flow} indicates a phenomenon which we never expected: the cone singularity structure is somehow stable even under smooth K\"ahler-Ricci flow. Namely, starting from the smooth metric $\omega_{\phi_{\epsilon}(\frac{a}{2})}$  which is close to a conical metric in Gromov-Hausdorff sense, the smooth flow $\omega_{\phi_{\epsilon}(t)}$ stays close to 
conical metrics at least up to finite time.  A natural and interesting question to ask is: what's the behaviour of the flow $\omega_{\phi_{\epsilon}(t)}$ (solution to \ref{Perturbed CKRF}) as $t\rightarrow \infty$?
\end{rmk}
The next theorem indicates that the conical flow constructed in Theorem 
1.2 in \cite{CYW} has the "smoothing" property. 
\begin{thm}\label{Bounding Ricci curvature when t>0}
There exists  a uniform $C$ as in Definition \ref{Convention on the constant} with the following properties.  Assumptions as in Theorem 1.2 of \cite{CYW}. The Ricci curvature of the (strong) CKRF satisfies 
\[|Ric|\leq Ct^{-1}\ \ \textrm{over}\ M\setminus D\ \textrm{for all}\ t\in [0,t_0],\]
where $t_0$ is a lower bound of the  existence time for the conical flow (\ref{Definition of CKRF}) and $t_0$ is small enough with respect to the initial metric $\omega_{\phi_0}$. 
Moreover, we have the following weighted Schauder estimate for $\frac{\partial \phi}{\partial t}$.
\[|\frac{\partial \phi}{\partial t}|^{(\star)}_{2+\alpha,1+\frac{\alpha}{2},\beta,M\times [0,t_0]}\leq C.\]
\end{thm}
The definition of the norm $|\ast |^{(\star)}_{2+\alpha,1+\frac{\alpha}{2},\beta,M\times [0,t_0]}$ can be found in section 2 of \cite{CYW}.

     For the proof of Theorem \ref{long time existence of the weak flow}, \ref{Existence in the weak flow0}, and \ref{Thm Smooth approximation of the conical flow}, we notice two beautiful  recent work by Guenancia-Paun \cite{GP} and Yao \cite{Yao}, where they prove independently
    the existence of weak conical K\"ahler-Einstein metrics under appropriate assumptions---these approaches have been
    taken up by others before, the new feature in their work is that the approximation stays uniformly quasi-isometric to the approximated model
metrics.  While we initially plan to use ideas from Yao where we need to do local cutting and pasting, we notice the beautiful construction of global barrier function in Guenancia-Paun which fits into what we want very nicely. So we end up adopting Guenancia-Paun's method, although we believe
    Yao's idea can be made to work as well.   \\
    
    Acknowledgements: This is a side project which grows out of a joint project with Prof  Xiuxiong Chen on the conical KRFs.
The author would like to thank Prof Chen for kindly suggesting this project and for his constant support
over years. The author also would like to thank Chengjian Yao for related discussions.

\section{Approximating the  initial metric.\label{section of approximating the intial metric}}

 We would like to make following convention on the constants and H\"older norms in this paper, similar to that of \cite{CYW}.
\begin{Def}\label{Convention on the constant} 
   \textit{Without further notice, the  "C" in each estimate  means a constant depending on the dimension $n$, the angle $\beta$, the background objects $(M,\omega_{0},L,h,D,\omega_{D})$, the $\alpha$ (and $\acute{\alpha}$ if any) in the same estimate or in the corresponding theorem (proposition, corollary, lemma), and finally the time $T$. Moreover, the "C" in different places might be different. }
   \end{Def}
   
    \begin{Def}\label{Def convention on Holder norms}(Convention on H\"older norms) For the various intrinsic H\"older norms with respect to $\omega_{\beta}$ (for example the $C^{,\alpha,\beta}$-norm  and its parabolic counterpart $C^{,\alpha,\frac{\alpha}{2},\beta}$), we mainly refer to section 2 of \cite{CYW} for the full definitions. The point is that, in this article we mainly  consider usual H\"older spaces and norms (without any additional "$\beta$" in the notations). The reason is that H\"older continuity with respect to $\omega_{\beta}$ is equivalent to H\"older continuity in the usual sense (apart from a difference of H\"older exponents). For a precise statement, see Lemma \ref{lem Holder continuity w.r.t different background metrics}.
 
 It's helpful to recall the definition of the  parabolic H\"older norm. For any parabolic cylinder  $B\times [T_1,T_2]$, the $C^{\alpha,\frac{\alpha}{2}}(B\times [T_1,T_2])$-norm is defined as
 \begin{eqnarray*}& &{|u|}_{\alpha,\frac{\alpha}{2}, B \times [T_1,T_2] } 
 \\&=&   \sup_{(x,t_1),(y,t_2)\in B \times [T_1,T_2]} \frac{|u(x,t_1)-u(y,t_2)|}{|x-y|^{\alpha}+|t_1-t_2|^{\frac{\alpha}{2}}}+ |u|_{0,B \times [T_1,T_2]},
 \end{eqnarray*}
 where $|u|_{0,B \times [T_1,T_2]}$ is  the usual $C^{0}-$norm. The $C^{,\alpha,\frac{\alpha}{2}}(B\times [T_1,T_2])-$space contains exactly those functions with finite ${|\ |}_{\alpha,\frac{\alpha}{2}, B \times [T_1,T_2] }-$norm. The global norms over $M$ or $M\times [T_{1},T_{2}]$ are defined by summing up the norms in each coordinate chart. This is very flexible: if we use the intrinsic coordinates of $\omega_{\beta}$ near $D$, we obtain the  $C^{,\alpha,\frac{\alpha}{2},\beta}(M\times [T_1,T_2])$-space; if we use the usual smooth coordinates near $D$, then we obtain the usual H\"older space
 $C^{\alpha,\frac{\alpha}{2}}(M\times [T_1,T_2])$.

 \end{Def}
   
  To construct the approximation flows, the first step is to construct an approximation of the initial metric in Theorem \ref{Existence in the weak flow0}. From now on we will repeatedly apply Theorem \ref{Harnack inequality elliptic}. \\

The initial metric $\omega_{\phi(0)}$ satisfies 
\begin{equation}
\omega^n_{\phi(0)}=\frac{e^{F(0)}}{|S|^{2-2\beta}}\omega^n_{0}.
\end{equation}
We try to smooth $F(0)$ first. We consider  the reference metric $\omega_{\epsilon}$, introduced in section 3.1 of \cite{GP} as 
\begin{equation}
\omega_{\epsilon}=\omega_{0}+\frac{1}{N}i\partial \bar{\partial}\chi_{\beta}(\epsilon+|S|^2),\ \textrm{where}\ \chi_{\beta}(\epsilon+y)=\beta\int^{y}_{0}\frac{(\epsilon+x)^{\beta}-\epsilon^{\beta}}{x}dx
\end{equation}
and N is a big enough number. As in \cite{GP}, We also denote 
\begin{equation*}
\Psi_{\epsilon,\rho}=\chi_{\rho}(\epsilon+|S|^2).
\end{equation*}
In application we always let $\rho$ to be small with respect to $\beta$, as in \cite{GP}.
\begin{thm}\label{smoothing the initial metric out} Suppose $\phi$ is a $C^{1,1,\beta}$ solution to the following equation \[\omega^n_{\phi}=\frac{e^{F}}{|S|^{2-2\beta}}\omega^n_{0}.\]
Suppose $F\in C^{2,\alpha,\beta}$. Then there exists an approximation sequence of $C^{4,\alpha^{\prime}}$ functions $\widehat{\phi}_{\epsilon}$ such that 
\begin{itemize}
\item $|\widehat{\phi}_{\epsilon}|_{\alpha^{\prime}}\leq C_{F}$,
\item $\frac{1}{C_{F}} \omega_{\epsilon}\leq \omega_{\widehat{\phi}_{\epsilon}}\leq C_{F}\omega_{\epsilon}$,
\item $\lim_{\epsilon\rightarrow 0}|\widehat{\phi}_{\epsilon}-\phi|_{\alpha^{\prime}}=0$,
\end{itemize}
 where $C_{F}$ only depends on $|F|_{2,\alpha,\beta}$ and the  data in Definition (\ref{Convention on the constant}). 

\end{thm}
\begin{proof}[Proof of Theorem \ref{smoothing the initial metric out}: ]
We smooth $F$ out by solving the following equation 
 \begin{equation}\label{Smoothing the initial cone metric using elliptic equation}
 \Delta_{\omega_{\epsilon}}F_{\epsilon}=\Delta_{\omega_{D}}F+a_{\epsilon},
 \end{equation}
 $a_{\epsilon}$ is chosen such that \begin{equation}\label{Harnack normalization constant a epsilon}
 \int_{M}(\Delta_{\omega_{D}}F+a_{\epsilon})\omega^n_{\epsilon}=0,
 \end{equation}
 and $F_{\epsilon}$ is normalized so that 
 \begin{equation}\int_{M}\frac{e^{F_{\epsilon}}}{(|S|^2+\epsilon)^{1-\beta}}\omega^n_{0}=1.
 \end{equation}
 Notice that (\ref{Harnack normalization constant a epsilon}) directly implies $\lim_{\epsilon \rightarrow}a_{\epsilon}=0$.
 Using the $\epsilon$-independent bounds on the global Sobolev and Poincare constants in Remark \ref{Sobolev and Poincare constant bounds}, we deduce the follow $L^{\infty}$ bound via Moser's iteration.
  \begin{equation}\label{equ Linfty bounds of Fepsilon}
 |F_{\epsilon}|_{L^{\infty}}\leq C.
 \end{equation}
 Using (\ref{equ Linfty bounds of Fepsilon}) and  the elliptic Harnack-inequality in Theorem \ref{Harnack inequality elliptic},  we estimate 
 \begin{equation}
 [F_{\epsilon}]_{\underline{\alpha}}\leq C.
 \end{equation}
 Therefore,  $F_{\epsilon}$ subconverge to some $F_{\infty}$ in $C^{\alpha^{\prime}}$, $\alpha^{\prime}<\underline{\alpha}$. Moreover 
 \begin{equation}
  \Delta_{\omega}F_{\infty}=\Delta_{\omega}F\ \textrm{over}\ M\setminus D.
\end{equation}   
Since $F_{\infty}\in C^{\alpha^{\prime}}$, then by Jeffres' trick in \cite{Jeffres}, we have 
$F_{\infty}=F$. The advantage of smoothing $F$ using equation (\ref{Smoothing the initial cone metric using elliptic equation}) is that by    Guenancia-Paun, the condition  \begin{equation} \label{Lower bound of the Laplacian of perturbed F}
\Delta_{\omega_{\epsilon}}F_{\epsilon}\geq -C
\end{equation}
 gives us the Laplacian estimate for the smoothing of the initial metric. Namely,  we smooth $\omega_{\phi(0)}$ by considering the following equation
\begin{equation}\label{Perturbed Elliptic equation}\omega_{\widehat{\phi}_{\epsilon}}^n=\frac{e^{F_{\epsilon}}}{(|S|^2+\epsilon)^{1-\beta}}\omega^n_{0},\ \sup_{M}\widehat{\phi}_{\epsilon}=0.
\end{equation}
By Yau, equation (\ref{Perturbed Elliptic equation}) admits a solution $\widehat{\phi}_{\epsilon}\in C^{4,\alpha^{\prime}}$. 
By the work of Guenancia-Paun in section 5.2 of \cite{GP}, Kolodziej's 
$L^{\infty}-$ estimate in \cite{Kolodziej} (also see Theorem 1.1 in \cite{DinewZhang} for a general statement), Theorem \ref{Harnack inequality elliptic},  and the condition (\ref{Lower bound of the Laplacian of perturbed F}), we obtain
\begin{eqnarray}
& &|\widehat{\phi}_{\epsilon}|_{\alpha^{\prime}}\leq C_{F}
\\& &\frac{1}{C_F} \omega_{\epsilon}\leq \omega_{\widehat{\phi}_{\epsilon}}\leq C_F\omega_{\epsilon}.
\end{eqnarray}
The proof is thus completed. 
\end{proof}
\section{Construction of the approximation flows and proofs of Theorem \ref{Existence in the weak flow0}. \label{Construction of the approximation flows}}
To construct the weak flow and approximate the CKRF, we first apply Theorem \ref{smoothing the initial metric out} to perturb the initial cone metric to $\omega_{\widehat{\phi}_{\epsilon}}$. Then we  consider $\omega_{\widehat{\phi}_{\epsilon}}$  as our new initial metric and consider  the following  approximation flows.
\begin{equation}\label{Perturbed CKRF}
 \left \{
\begin{array}{ccr}
 & \frac{\partial \bar{\phi}_{\epsilon}}{\partial t}=\log\frac{\omega_{\bar{\phi}_{\epsilon}}^n}{\omega_0^n}+\beta\bar{\phi}_{\epsilon}+h+(1-\beta)\log(|S|^2+\epsilon),\ t\in[0,T].\\
   & \bar{\phi}_{\epsilon}(0)= \widehat{\phi}_{\epsilon} \ \textrm{when} \  t=0.\ \\
\end{array} \right.
\end{equation}
Now we would like to change the reference metric to $\omega_{\epsilon}$. Writing \[
\widehat{V}_{\epsilon}=h+\log\frac{\omega^n_{\epsilon}(|S|^2+\epsilon)^{1-\beta}}{\omega^n_{0}}+\frac{\beta}{N}\Psi_{\epsilon,\beta},\] we change the flow equation to the following. 
\begin{equation}\label{Perturbed CKRF before rescaling}
 \left \{
\begin{array}{ccr}
 & \frac{\partial \phi_{\epsilon}}{\partial t}=\log\frac{\omega^n_{\phi_{\epsilon}}}{\omega^n_{\epsilon}}+\beta\phi_{\epsilon}+\widehat{V}_{\epsilon},\ t\in[0,T]. \\
   & \phi_{\epsilon}(0)= \widehat{\phi}_{\epsilon}-\frac{\beta}{N}\Psi_{\epsilon,\beta} \ \textrm{when} \  t=0.\\
\end{array} \right.
\end{equation}

\begin{lem}\label{Parabolic Laplacian estimate before rescaling}There exists a constant $C$ in the sense of Definition \ref{Convention on the constant} with the following properties. On the perturbed flow (\ref{Perturbed CKRF before rescaling}), the following estimates  hold. 
\begin{itemize}
\item $|\phi_{\epsilon}|_{\alpha,\frac{\alpha}{2}}\leq C$,
\item $\frac{1}{C} \omega_{\epsilon}\leq \omega_{\phi_{\epsilon}}\leq C\omega_{\epsilon}$, $|\frac{\partial \phi_{\epsilon}}{\partial t}|_{L^{\infty}}\leq C$.
\end{itemize}
\end{lem}
\begin{proof}[ Proof of Lemma \ref{Parabolic Laplacian estimate before rescaling}:]

Step1: First we show that  $|\frac{\partial \phi_{\epsilon}}{\partial t}|_{L^{\infty}}\leq C$. This is directly implied by the maximal principle and the bound on $|\frac{\partial \phi_{\epsilon}}{\partial t}|_{t=0}$. The bound 
on $|\frac{\partial \phi_{\epsilon}}{\partial t}|_{t=0}$ directly follows from the properties of our approximating intial metrics. Namely from (\ref{Perturbed Elliptic equation}) we have
\begin{eqnarray}
\nonumber & &\frac{\partial \phi_{\epsilon}}{\partial t}|_{t=0}
\nonumber \\&=&\log\frac{\omega^n_{\phi_{\epsilon}}}{\omega^n_{\epsilon}}(0)+\beta\phi_{\epsilon}(0)+\widehat{V}_{\epsilon}(0)
\nonumber\\&=&\log \frac{e^{F_{\epsilon}}\omega^n_{0}}{(|S|^2+\epsilon)^{1-\beta}\omega^n_{\epsilon}}(0)
+\beta\phi_{\epsilon}(0)+\widehat{V}_{\epsilon}(0).
\end{eqnarray}
Thus by Theorem \ref{smoothing the initial metric out}  we obtain \begin{equation*}|\frac{\partial \phi_{\epsilon}}{\partial t}|_{0,M,t=0}\leq C.\end{equation*} Therefore by maximal principle we have  \begin{equation}\label{Appro res: Bound on the dphidt of the perturbed rescaled flow}
|\frac{\partial \phi_{\epsilon}}{\partial t}|_{0,M}\leq C.
\end{equation}\\

Step 2. Now we turn to the spacewise second order  estimate. By the Siu Bochner technique in \cite{Siu}(the reader could also see \cite{GP}) and the flow equation (\ref{Perturbed CKRF before rescaling}), denote \begin{equation*}h_{\epsilon}=-\beta\phi_{\epsilon}-\widehat{V}_{\epsilon},
\end{equation*}  we  derive the following parabolic Siu-Bochner formula.
 \begin{eqnarray}\label{Parabolic Siu Bochner technique}
 & &(\Delta_{\phi_{\epsilon}}-\frac{\partial }{\partial t})\log {tr}_{\omega_{\epsilon}}\omega_{\phi_{\epsilon}}\nonumber
\\&\geq &\frac{1}{tr_{\omega_{\epsilon}}\omega_{\phi_{\epsilon}}} 
\{\Delta_{\omega_{\epsilon}}h_{\epsilon}+\Sigma_{i\leq l}(\frac{\lambda_i}{\lambda_{l}}+\frac{\lambda_l}{\lambda_{i}}-2)R_{i\bar{i}l\bar{l}}(w)\},
 \end{eqnarray}
 where $w$ is the geodesic coordinates of $\omega_{\epsilon}$ which diagonalize $\omega_{\phi_{\epsilon}}$ with respect to $\omega_{\epsilon}$, and $\lambda_i$ are the eigenvalues of 
 $\omega_{\phi_{\epsilon}}$ with respect to $\omega_{\epsilon}$.
 
 We then consider  the barrier function $\Psi_{\epsilon,\rho}$ (for sufficiently small $\rho$) of Guenancia and Paun in \cite{GP}. Namely, for the sake of a self-contained proof, we quote in the following two beautiful identities from \cite{GP}  at any point $p$ near  $D$ (which do not depend on the  flow).
 \begin{itemize}
 \item Equation ($\star$) in page 8 of  \cite{GP}:
 \begin{equation}\label{GP subharmonicity of Psi}
 \Delta_{\omega_{\phi_{\epsilon}}}\Psi_{\epsilon,\rho}\geq -Ctr_{\omega_{\phi_{\epsilon}}}\omega_{\epsilon}+C\Sigma_{i=1}^{n}\frac{1}{(|S|^2+\epsilon)^{1-\rho}}|\frac{\partial z}{\partial w_{i}}|^2\frac{1}{\lambda_i}.
 \end{equation}
  \item Curvature estimate  in page 8 of \cite{GP}:
 \begin{eqnarray}\label{GP the lower bound on the curvature}
 & &\frac{1}{\Sigma_{k}\lambda_{k}}
\{\Delta_{\omega_{\epsilon}}h_{\epsilon}+\Sigma_{i\leq l}(\frac{\lambda_i}{\lambda_{l}}+\frac{\lambda_l}{\lambda_{i}}-2)R_{i\bar{i}l\bar{l}}(w)\}\nonumber
\\&\geq & -C\Sigma_{i=1}^{n}\frac{1}{(|S|^2+\epsilon)^{1-\rho}}|\frac{\partial z}{\partial w_{i}}|^2\frac{1}{\lambda_i} \nonumber
\\&-&\frac{1}{\Sigma_{k}\lambda_{k}}\{\Sigma_{i\leq l}(\frac{\lambda_i}{\lambda_{l}}+\frac{\lambda_l}{\lambda_{i}})\}-C.
 \end{eqnarray}
 \end{itemize}
 
 Then, (\ref{Parabolic Siu Bochner technique}), (\ref{GP subharmonicity of Psi}), and (\ref{GP the lower bound on the curvature}) imply the following estimate for sufficiently big numbers $A$ and $B$ over the whole $M$.
 \begin{eqnarray}\label{Siu Bochner consequence}
 & &(\Delta_{\phi_{\epsilon}}-\frac{\partial }{\partial t})\{\log {tr}_{\omega_{\epsilon}}\omega_{\phi_{\epsilon}}+B\Psi_{\epsilon,\rho}-A\phi_{\epsilon}\}\nonumber
\\&\geq & {tr}_{\omega_{\phi_{\epsilon}}}\omega_{\epsilon}+A \frac{\partial \phi_{\epsilon}}{\partial t}-C.
 \end{eqnarray}
 
 (\ref{Appro res: Bound on the dphidt of the perturbed rescaled flow}) and (\ref{Siu Bochner consequence}) indicate that,  at the interior maximum of 
 \begin{equation}\label{eq Bound of the laplacian at the maximum of Siu quantity }\log {tr}_{\omega_{\epsilon}}\omega_{\phi_{\epsilon}}+B\Psi_{\epsilon,\rho}-A\phi_{\epsilon},\  \textrm{we have}\ 
 {tr}_{\omega_{\phi_{\epsilon}}}\omega_{\epsilon}\leq C.
 \end{equation}
 In terms of the eigenvalues, we have
 \begin{equation}\label{Inverse trace bound}\Sigma_{i}\frac{1}{\lambda_i}\leq C.
 \end{equation}
 By (\ref{Appro res: Bound on the dphidt of the perturbed rescaled flow}) we have  \begin{equation}\label{DET bound}
 \frac{1}{C}\leq \prod_{k}\lambda_k\leq C.
 \end{equation} Hence
 \begin{eqnarray}\label{Bounding trace by inverse trace}
 & &(\Sigma_{i}\frac{1}{\lambda_i})^{n-1}\nonumber
 \\&\geq & \Sigma_{i}\frac{1}{\lambda_1..\widehat{\lambda_{i}},...\lambda_n}
 = \frac{\Sigma_{i}\lambda_i}{\prod_{k}\lambda_k}
 \geq C\Sigma_{i}\lambda_i.
 \end{eqnarray}
 Combining (\ref{Bounding trace by inverse trace}), (\ref{eq Bound of the laplacian at the maximum of Siu quantity }), and (\ref{Inverse trace bound}), we end up with 
 \begin{equation}
 \sup {tr}_{\omega_{\epsilon}}\omega_{\phi_{\epsilon}}\leq C.
 \end{equation}
 Using (\ref{DET bound}) again, we get $\sup {tr}_{\omega_{\phi_{\epsilon}}}\omega_{\epsilon}<C$.\\
 
 Step 3: In Step 2 we assume that  ${tr}_{\omega_{\epsilon}}\omega_{\phi_{\epsilon}}$ attains interior maximum. On the other hand, suppose ${tr}_{\omega_{\epsilon}}\omega_{\phi_{\epsilon}}$ attain maximum when $t=0$, then the second item in Theorem \ref{smoothing the initial metric out}  implies our desired bound. Thus item 2 in Lemma \ref{Parabolic Laplacian estimate before rescaling} is proved. 
 
 Step 4: To prove Item 1, it suffices to use item 2 and the Harnack inequality in Theorem \ref{Harnack inequality elliptic}. Notice that we automatically have the $C^{0}$-estimate via the bound on $|\frac{\partial \phi_{\epsilon}}{\partial t}|_{0}$ and the $C^{0}$-bound on the initial potential $\phi_{\epsilon}(0)$.
 
 The proof is complete.

\end{proof}

Using Lemma \ref{Parabolic Laplacian estimate before rescaling} and Theorem \ref{smoothing the initial metric out}, $\phi_{\epsilon}$ sub converges to a
$\phi_{\infty}$ such that 
\begin{displaymath}\left \{
\begin{array}{ccr}
 & \frac{\partial \phi_{\infty}}{\partial t}=\log\frac{\omega_{\phi_{\infty}}^n}{\omega_0^n}+\beta\phi_{\infty}+h+(1-\beta)\log(|S|^2),\  t\in[0,T].\\
   & \phi_{\infty}(0)= \phi(0) \ \textrm{when} \  t=0.\\\
   & \phi_{\infty}\in C^{\alpha,\frac{\alpha}{2}}[0,T];\ \frac{1}{C} \omega_{D}\leq \omega_{\phi_{\infty}}\leq C\omega_{D}.\\
\end{array} \right.
\end{displaymath}

To show that the perturbation really converges back to the original conical flow, we need the following lemma on the uniqueness of the weak conical flows. 
\begin{lem}\label{uniqueness of weak flows}Suppose $\phi_i,i=1,2$ are two weak conical flows :
\begin{displaymath}\left \{
\begin{array}{ccr}
 & \frac{\partial \phi_{i}}{\partial t}=\log\frac{\omega_{\phi_{i}}^n}{\omega_0^n}+\beta\phi_{i}+h+(1-\beta)\log(|S|^2)\ \textrm{over}\ M\setminus D,\\
   & \phi_{i}(0)= \phi(0) \ \textrm{when} \  t=0,\\
   & \phi_{i}\in C^{\alpha,\frac{\alpha}{2}}(M\times [0,T]).\\
   & \frac{1}{C} \omega_{D}\leq \omega_{\phi_{i}}\leq C\omega_{D}, \ |\frac{\partial \phi_{i}}{\partial t}|_{L^{\infty}}\leq C\ \textrm{over}\ M\setminus D.
\end{array} \right.
\end{displaymath}

Then $\phi_1=\phi_2$.

\end{lem}
\begin{proof}
We again employ Jeffres' trick in the parabolic case, adapted to our setting. Consider $\widehat{\phi}_1=\phi_1+a|S|^{2p}$. Then we compute
\begin{eqnarray}\label{Appro res: difference of two solutions via Jeffres perturbation}
\frac{\partial \widehat{\phi}_1}{\partial t}&=&\log\frac{(\omega+i\partial \bar{\partial}\phi_1)^{n}}{\omega_0^n}\nonumber
+\beta\widehat{\phi}_{1}-a\beta |S|^{2p}+h+(1-\beta)\log(|S|^2).
\end{eqnarray}
Denote $v=\widehat{\phi}_1-\phi_2$ and $\underline{\Delta}=\int_{0}^{1}g^{i\bar{j}}_{b\phi_1+(1-b)\phi_2}
\frac{\partial^{2}}{\partial z_i\partial\bar{z_j} }db$, we compute from 
(\ref{Appro res: difference of two solutions via Jeffres perturbation}) that  
\begin{eqnarray}\label{equation of the difference of two weak flows}
\frac{\partial v}{\partial t}=\underline{\Delta}v-a\underline{\Delta}|S|^{2p}
+\beta v-a\beta |S|^{2p}.
\end{eqnarray}
The following is due to Jeffres \cite{Jeffres}.
\begin{clm}\label{clm supremum attained away from D} When $p<\frac{\alpha \beta}{2}$, the spacewise maximum of $\widehat{\phi}_1$ and $v$ are attained away from $D$.
\end{clm}
 We only prove it for $\widehat{\phi}_{1}$, the others are similar. Was this claim not true,   let  the spacewise maximum of $\widehat{\phi}_{1}$ be attained at $q\in D$, near $q$ we have a holomorphic chart such that $q$ corresponds to the origin $0$, and $|S|^{2}=h |z|^{2}$, $h$ is the metric of the line bundle of $D$ ( $\frac{1}{C}\leq h\leq C$). Suppose in this chart we have
 \begin{equation}\label{equ phi1 attain max on D}
 \phi_{1}(z,0)+\epsilon |S|^{2p}(z,0)-\phi_{1}(0,0)\leq 0.
\end{equation}  
 
 Since $\phi_{1}\in C^{\alpha}$ spacewisely, we compute, 
 \begin{eqnarray*}
& & \frac{\phi_{1}(z,0)+\epsilon |S|^{2p}(z,0)-\phi_{1}(0,0)}{|z|^{\alpha}}=\frac{\phi_{1}(z,0)-\phi_{1}(0,0)}{|z|^{\alpha}}
+\frac{\epsilon h^{p} |z|^{2p}}{|z|^{\alpha}}
\\&\geq & -[\phi_{1}]_{\alpha}+\epsilon h^{p} |z|^{-(\alpha-2p)}
\\& \geq & 1 \ \textrm{when}\ z\ \textrm{is sufficiently close to}\ 0,\ p<\frac{\alpha\beta}{2}. 
 \end{eqnarray*}
  This contradicts (\ref{equ phi1 attain max on D}). The proof of Claim \ref{clm supremum attained away from D} is complete. Actually it sufficies to require $2p<\alpha$. The reason of requiring the stronger condition $2p<\alpha\beta$ is that it even works more generally for $\phi_{1}\in C^{,\alpha,\beta}$ (the intrinsic H\"older space of $\omega_{\beta}$), thus we can avoid any confusion.

 Furthermore, we have 
\begin{eqnarray}\label{formula of i ddbar of norm of S}
& &i\partial \bar{\partial}|S|^{2p}\nonumber
\\&=& p^2|S|^{2p}\partial \log |S|^2\wedge \bar{\partial }\log |S|^2+
p|S|^{2p}i\partial \bar{\partial}\log |S|^2.
\end{eqnarray}
By the second-order estimate in the assumptions, we have the following estimate.
\begin{equation}\label{eq bound on g inverse}
|g^{i\bar{j}}_{b\phi_1+(1-b)\phi_2}|\leq C,\ \textrm{where the basis is}\ (\frac{\partial}{\partial z_i},\ i=1...n).
\end{equation}
Then away from the divisor, $g^{i\bar{j}}_{b\phi_1+(1-b)\phi_2}$ is at least $C^{\alpha}$. From (\ref{formula of i ddbar of norm of S}) we compute over $M\setminus D$ that
\begin{eqnarray}\label{Laplacian of norm of S is bounded from below}
\nonumber& &\underline{\Delta}|S|^{2p}
\nonumber\\&\geq & 
p|S|^{2p}g^{i\bar{j}}_{b\phi_1+(1-b)\phi_2}
\frac{\partial^{2}\log |S|^2}{\partial z_i\partial\bar{z_j} }
\nonumber \\&\geq & -C|pg^{i\bar{j}}_{b\phi_1+(1-b)\phi_2}
\Theta_{h,i\bar{j}}|
\nonumber\\&\geq & -C,
\end{eqnarray}
where $\Theta_{h}$ is the smooth curvature form of $(L,h)$.

Then, by (\ref{equation of the difference of two weak flows}) and Proposition 2.23 in \cite{MorganTian},     we deduce
\begin{equation}\frac{\partial \sup v}{\partial t}\leq aC
+\beta \sup v,\end{equation}
in the sense of forward difference quotients.
Using $v(0)=a|S|^{2p}\leq aC$ and Proposition 2.23 in \cite{MorganTian}again, we obtain 
\begin{equation}\sup v\leq [\sup v(0)+aC]e^{\beta t}-aC\leq aCe^{\beta T}.
\end{equation}
Thus  let $a$ tend to 0,  we end up with $\phi_1\leq \phi_2$. By the same reason we have 
$\phi_2\leq \phi_1$, then $\phi_1= \phi_2$.
\end{proof}
\begin{proof}[ Proof of Theorem \ref{Existence in the weak flow0}:]
By letting $\epsilon\rightarrow 0$ in Lemma \ref{Parabolic Laplacian estimate before rescaling} and flow (\ref{Perturbed CKRF}), notice that our time $T$ can be arbitrarily large, then
Theorem \ref{Existence in the weak flow0} is a direct corollary of 
Lemma \ref{Parabolic Laplacian estimate before rescaling} and \ref{uniqueness of weak flows}.

\end{proof}
\section{Local Harnack inequality.\label{section Local harnack inequality}}
  In this section we prove Theorem \ref{Harnack inequality elliptic} by proving the  harder  Theorem \ref{Harnack inequality parabolic}. Theorem \ref{Harnack inequality elliptic} is all we need to prove the results in the introduction. 
 
    \begin{thm}\label{Harnack inequality elliptic}
Suppose  $\omega$ is a weak conical-K\"ahler metric and  $ \frac{\omega_{D}}{C}\leq \omega\leq C\omega_{D}$, or 
$\omega$ is $C^{\alpha}$ over the whole $M$ (across $D$) in the usual sense and 
$\frac{\omega_{\epsilon}}{C}\leq \omega\leq C\omega_{\epsilon}$ for some $0<\epsilon\leq 1$.

 Suppose  $u$ is  a bounded weak solution to  
\[\Delta_{\omega}u
=f\ \textrm{over}\ M,\]
then there exists a $\alpha^{\prime}>0$ such that 
\[[u]_{\alpha^{\prime},M}\leq C(|u|_{0,M}+|f|_{0,M}).\]
\end{thm} 
     \begin{thm}\label{Harnack inequality parabolic0} Suppose $\omega_t$ is a (strong) conical K\"ahler-Ricci flow, or the perturbed smooth flow (\ref{Perturbed CKRF}) over $[0,T]$.   Suppose  $u$ is  a bounded weak solution to  
\[\frac{\partial }{\partial t}u
=\Delta_{\omega_{t}}u+f\ \textrm{over}\ M\times [0,T].\]
Then for all $\delta\in (0,T)$, there exists a $\alpha^{\prime}>0$ and a $C(\delta)$ in the sense of Definition \ref{Convention on the constant} such that
\[[u]_{\alpha^{\prime},\frac{\alpha^{\prime}}{2},M\times [\delta,T]}\leq C(\delta)(|u|_{0,M\times [0,T]}+|f|_{0,M\times [0,T]}).\]
\end{thm}
  \begin{proof}[ Proof of Theorem \ref{Harnack inequality parabolic0} and \ref{Harnack inequality elliptic}:] Theorem \ref{Harnack inequality parabolic0} is  directly implied by Theorem \ref{Harnack inequality parabolic} and Lemma \ref{Bound on the Scalar Curvature}. Theorem \ref{Harnack inequality elliptic} is implied  by Theorem \ref{Harnack inequality parabolic} directly.
  \end{proof} 
\begin{lem}\label{Existence of Polar coordinates of the omega beta epsilon}
For any $\epsilon > 0$ and any point $p\in D$, there exists a canonical polar coordinate $\imath_{\epsilon}$ such that 
\begin{equation}
\imath_{\epsilon}^{\star}\{\frac{\beta^2}{(|z|^2+\epsilon)^{1-\beta}}dz\otimes d\bar{z}\}=ds^2+a_{\epsilon}s^2d\theta^2,\ p\ \textrm{is the origin in these coordinates},
\end{equation}
where $a_{\epsilon}$ is a smooth function of $s$ ($s\in [0,r_0)$ for $r_0$ sufficiently small), and $\beta^2<a_{\epsilon}\leq 1$. In particular, we have 
\begin{equation}
\beta^2\omega_{E,\epsilon}< \imath_{\epsilon}^{\star}\omega_{\beta,\epsilon}\leq \omega_{E,\epsilon},
\end{equation}
where $\omega_{E,\epsilon}=ds^2+s^2d\theta^2+\Sigma_{i=1}^{n-1}du_i\otimes d\bar{u}_i $ is the Euclidean metric in the coordinate $\imath_{\epsilon}$. 
\end{lem}
\begin{proof}
Actually the proof is quite elementary, since this fact is very important we include the full detail here. Let $\rho=|z|$. Define $s$ as  
\begin{equation}\label{ODE for the s in coordinates}
\frac{ds}{d\rho}=\frac{\beta}{(\rho^2+\epsilon)^{\frac{1-\beta}{2}}},\
s(0)=0.
\end{equation}

Then 
\[\frac{\beta^2}{(|z|^2+\epsilon)^{1-\beta}}dz\otimes d\bar{z} =ds^2+a_{\epsilon}s^2d\theta^2,\]
where 
\begin{equation}
a_{\epsilon}=\frac{\beta^2\rho^2}{(\rho^2+\epsilon)^{1-\beta}s^2}.
\end{equation}

From  (\ref{ODE for the s in coordinates}) we have \begin{equation}\label{Weigthted ODE of the coordinate s}
\frac{d[(\rho^2+\epsilon)^{\frac{1-\beta}{2}}s]}{d\rho}=\beta+
\frac{(1-\beta)\rho s}{(\rho^2+\epsilon)^{\frac{1+\beta}{2}}}\geq\ \beta.
\end{equation}
Then \[\beta\rho\leq (\rho^2+\epsilon)^{\frac{1-\beta}{2}}s.\]
Hence
\[a_{\epsilon}= \frac{\beta^2\rho^2}{(\rho^2+\epsilon)^{1-\beta}s^2}\leq 1.\]
Now we would like to study the uniform lower bound of $a_{\epsilon}$. 
Using (\ref{Weigthted ODE of the coordinate s}), denote $v=(\rho^2+\epsilon)^{\frac{1-\beta}{2}}s$,  we compute 
\begin{eqnarray*}
\frac{dv}{d\rho}&<&\beta+
\frac{(1-\beta)(\rho^2+\epsilon)^{\frac{1}{2}} s}{(\rho^2+\epsilon)^{\frac{1+\beta}{2}}}
\\&=& \beta+
(1-\beta)v(\rho^2+\epsilon)^{-\frac{1}{2}}
\\&<&   \beta+
(1-\beta)\frac{v}{\rho}.
\end{eqnarray*}

Then 
\begin{eqnarray*}
\frac{d}{d\rho}(\frac{v}{\rho})&=&-\frac{v}{\rho^2}+\frac{1}{\rho}\frac{dv}{d\rho}
\\&\leq &-\frac{v}{\rho^2}+\frac{\beta}{\rho}+(1-\beta)\frac{v}{\rho^2}
\\&=& \frac{\beta}{\rho}-\frac{\beta v}{\rho^2}
\\&=& \frac{\beta}{\rho}(1-\frac{v}{\rho}).
\end{eqnarray*}
Denote $u=\frac{v}{\rho}$, we get 
\begin{equation}\label{ODE of u in the coordinates}
\frac{d u}{d\rho}<\frac{\beta}{\rho}(1-u).
\end{equation}
Furthermore, from  (\ref{ODE for the s in coordinates}) we have $u(0)=\beta$. Then simple comparison implies 
\begin{equation}\label{Coordinates u less than rho} u<1 \ \textrm{for all} \ \rho.
\end{equation} 
To be precise, if $u(\rho_0)=1$ for some $\rho_0$, then take  $\rho_0$ as the first one among those $\rho$ of which $u(\rho)=1$, then we deduce $\frac{d u}{d\rho}(\rho_0)\geq 0$, which contradicts (\ref{ODE of u in the coordinates}).\\

Then it's easy to see from (\ref{Coordinates u less than rho}) that 
\[a_{\epsilon}=\frac{\beta^2}{u^2}>\beta^2.\]
\end{proof}

When $\epsilon=0$, the coordinate in Lemma \ref{Existence of Polar coordinates of the omega beta epsilon} is exactly the polar coordinate of $\omega_{\beta}$. Denote $|x-y|_{(\beta,\epsilon)}$ as the distance between $x,y$ in the polar coordinate in Lemma \ref{Existence of Polar coordinates of the omega beta epsilon}, and $|x-y|_{holo}$  as the distance in the holomorphic (smooth) coordinates. Comparing the 2 distances gives the equivalence of H\"older continuities with respect to the 2 different model metrics. To be precise, let $C^{,\alpha,(\beta,\epsilon)}(M)$ be the H\"older space of exponent $\alpha$ with respect to the distance $|\ |_{(\beta,\epsilon)}$, and $C^{,\alpha,\frac{\alpha}{2}, (\beta,\epsilon)}(M)$ be its parabolic counterpart, the following is true.

\begin{lem}\label{lem Holder continuity w.r.t different background metrics} Suppose $\epsilon\in [0, \frac{1}{100}]$. Given any 2 points $x,\ y$ such that  $|x|_{(\beta,\epsilon)},\ |y|_{(\beta,\epsilon)}\leq 1$, the following is true 
\begin{equation}\label{equ in lem Holder continuity w.r.t different background metrics}
 C|x-y|_{holo}\leq  |x-y|_{(\beta,\epsilon)}\leq C|x-y|_{holo}^{\beta}.
   \end{equation}
\end{lem}
Consequently 
\begin{itemize}
\item  $C^{,\alpha,(\beta,\epsilon)}(M)$ embeds continuously into $C^{\alpha\beta}(M)$, $C^{\alpha}(M)$ embeds continuously into $C^{,\alpha,(\beta,\epsilon)}(M)$
\item  $C^{,\alpha, \frac{\alpha}{2}, (\beta,\epsilon)}(M)$ embeds continuously into $C^{\alpha\beta, \frac{\alpha}{2}}(M)$, $C^{\alpha,\frac{\alpha}{2}}(M)$ embeds continuously into $C^{,\alpha,\frac{\alpha}{2},(\beta,\epsilon)}(M)$
\end{itemize}

\begin{proof}[Proof of Lemma \ref{lem Holder continuity w.r.t different background metrics}:] The proof is elementary, and is obvious when $\epsilon=0$.  The embedding results are straightforward by Definition \ref{Def convention on Holder norms} and (\ref{equ in lem Holder continuity w.r.t different background metrics}). To prove (\ref{equ in lem Holder continuity w.r.t different background metrics}),  it sufficies to prove the following claim.
\begin{clm}\label{clm comparison of the distances}  
   \begin{equation}
 C|\rho_{x}-\rho_{y}|\leq  |s_{x}-s_{y}|\leq C|\rho_{x}-\rho_{y}|^{\beta}.
   \end{equation}
   In particular we have 
    \begin{equation}
 C\rho_{x}\leq  s_{x}\leq C\rho_{x}^{\beta},\  C\rho_{y}\leq  s_{y}\leq C\rho_{y}^{\beta}. 
   \end{equation}
\end{clm}
Recall  the following  
\begin{eqnarray*}
& & |x-y|_{(\beta,\epsilon)}\approx |s_{x}-s_{y}|+|\sin \frac{\theta_{x}-\theta_{y}}{2}|\sqrt{s_{x}s_{y}}+|x_{T}-y_{T}|,
\\& & |x-y|_{holo}\approx |\rho_{x}-\rho_{y}|+|\sin \frac{\theta_{x}-\theta_{y}}{2}|\sqrt{\rho_{x}\rho_{y}}+|x_{T}-y_{T}|,
\end{eqnarray*}
where $x_{T},\ y_{T}$ are the tangential component of $x,y$ along the singularity. Then Claim \ref{clm comparison of the distances} directly implies (\ref{equ in lem Holder continuity w.r.t different background metrics}).

Next we prove Claim \ref{clm comparison of the distances}. Since $\frac{ds}{d\rho}\leq C\rho^{\beta-1}$, we deduce \begin{equation}
|s_{x}-s_{y}|\leq C|\rho_{x}^{\beta}-\rho_{y}^{\beta}| \leq C|\rho_{x}-\rho_{y}|^{\beta}.
\end{equation}

To obtain lower bound, notice when $\rho>10\sqrt{\epsilon}$, we have
$\frac{ds}{d\rho}\geq C\rho^{\beta-1}$, then  
\begin{equation}
|s_{x}-s_{y}|\geq C|\rho_{x}^{\beta}-\rho_{y}^{\beta}|\geq C|\rho_{x}-\rho_{y}|.
\end{equation}
When $\rho\leq 10\sqrt{\epsilon}$, $\frac{\epsilon^{\frac{\beta-1}{2}}}{C}\leq \frac{ds}{d\rho}\leq C\epsilon^{\frac{\beta-1}{2}}$. Hence 
\begin{equation}
|s_{x}-s_{y}|\geq \frac{\epsilon^{\frac{\beta-1}{2}}}{C}|\rho_{x}-\rho_{y}|.
\end{equation}
Thus when $\rho_{x},\ \rho_{y}\leq 1$, we deduce
   \begin{equation}
|s_{x}-s_{y}|\geq C|\rho_{x}-\rho_{y}|.
\end{equation}
The proof of Claim \ref{clm comparison of the distances} is complete. 
\end{proof}
\begin{rmk}\label{Sobolev and Poincare constant bounds}Actually Lemma \ref{Existence of Polar coordinates of the omega beta epsilon} implies the bound for the Poincare and Sobolev constants in the following sense. For any model metric $\omega$, denote
 \[E_{\omega,\lambda}=\{\omega^{\prime}|r^{\alpha}[\omega^{\prime}]_{\alpha, M\setminus T_{r}(D)}\leq \lambda,\ \frac{\omega}{\lambda}\leq  \omega^{\prime}\leq \lambda \omega\}.\] With respect to the global perturbed metric $\omega_{\epsilon}$,  using Lemma \ref{Existence of Polar coordinates of the omega beta epsilon}, it's quite straight forward to show the  global and local Sobolev constants $C_{S,\epsilon}, C^{\star}_{S,\epsilon}$ for all the metrics in $E_{\omega_{\epsilon},\lambda}$ are uniformly bounded from above independent of $\epsilon$. On the Poincare inequality, the global and local Poincare constants $ C_{P,0}, C^{\star}_{P,0}$ of $E_{\omega_{D},\lambda}$ are uniformly bounded from above. Moreover, using very simple counter-proofs based on the Rellich-Kondrachov compact-imbedding theorem, we deduce 
 that both  the local Poincare  constants $C^{\star}_{P,\epsilon}$ and the global  Poincare  constants $C_{P,\epsilon}$ of $E_{\omega_{\epsilon},\lambda}$ are bounded from  above independent of $\epsilon$. These are necessary for doing the Nash-Moser iteration and the proofs of the Harnack inequalities in   \cite{LSU} and \cite{GT}.
\end{rmk}

\begin{thm}\label{Harnack inequality parabolic}There exists a constant $C$ in the sense of Definition \ref{Convention on the constant} with the following properties. Suppose  $\omega_t$ is
a time-differentiable family of K\"ahler metrics  which is $C^{\alpha}$ away from $D$. Suppose  
\begin{enumerate}
\item $\omega_{t}$ are weak conical-K\"ahler metrics and $ \frac{\omega_{D}}{C}\leq \omega_t\leq C\omega_{D}$ for all $t\in [0,T]$, or 
$\omega_t$ is $C^{\alpha}$ over the whole $M$ (across $D$) in the usual sense and 
$\frac{\omega_{\epsilon}}{C}\leq \omega_t\leq C\omega_{\epsilon}$ for all $t\in [0,T]$ and some $0<\epsilon\leq 1$ ;
\item  
$\frac{\partial }{\partial t}dvol_t\leq Cdvol_t.$
\end{enumerate}
 Suppose  $u$ is  a bounded weak solution to  
\[\frac{\partial }{\partial t}u
=\Delta_{\omega_{t}}u+f\ \textrm{over}\ M\times [0,T].\]
Then for all $\delta\in (0,T)$, there exists a $\alpha^{\prime}>0$ and $C(\delta)$ such that 
\[{[u]}_{\alpha^{\prime},\frac{\alpha^{\prime}}{2},M\times [\delta,T]}\leq C(\delta)(|u|_{0,M\times [0,T]}+|f|_{0,M\times [0,T]}).\]
\end{thm}
   \begin{proof} [ Proof of Theorem \ref{Harnack inequality parabolic}:] With Lemma \ref{Existence of Polar coordinates of the omega beta epsilon}, actually the proof is quite straight forward. The only possible problem is the H\"older estimate near  $D$. Notice that in the coordinate $\imath_{\epsilon}$, $\omega_{t}$ is quasi-isometric  to the Euclidean metric. 
   
 It's well known that the  integration by parts holds true (c.f \cite{WYQ}).  For the reader's convenience we include  the proof of this fact here. We only consider the case when $\omega_{t}$ is weakly conic i.e $ \frac{\omega_{D}}{C}\leq \omega_t\leq C\omega_{D}$. Suppose $B$ is a ball with nonempty intersection with $D$,  and  $u$ is locally $C^{2}$ function defined over $B\setminus D$, $|\nabla  u|\in L^2[B]$,  Suppose  $v\in C_{c}^{0}[B]\cap C^{1}(B\setminus D)$, $|\nabla v|\in L^2(B)$. The integration by parts formula  we want to  show is:   
\begin{equation}\label{equ int by parts}
\lim_{\epsilon_{k}\rightarrow 0}\int_{B\setminus T_{\epsilon_{k}}(D)}v \Delta u\ \omega_{t}^{n}=-\int_{B}\nabla v\cdot\nabla u\ \omega_{t}^{n}\ \textrm{for some sequence}\ \epsilon_{k}\ \rightarrow 0.
\end{equation}

Proof of (\ref{equ int by parts}):  
We compute for any $\epsilon>0$ that 
\begin{eqnarray*} \int_{B\setminus T_{\epsilon}(D)}v \Delta u \ \omega_{t}^{n}
= -\int_{B\setminus T_{\epsilon}(D)}\nabla v\cdot \nabla u\ \omega_{t}^{n}+\int_{\partial T_{\epsilon}(D)}v(\nabla u\cdot n)dA.
\end{eqnarray*}
It suffices to show there exists a sequence $\epsilon_{k}\rightarrow 0$ such that 
\begin{equation}\label{equ lim of v nabla u times n vanishes}
\lim_{k \rightarrow \infty}\int_{\partial T_{\epsilon_{k}}(D)}v(\nabla u\cdot n)dA=0.
\end{equation}

Let $F=|\nabla u|$, extend $F$ to be 0 outside $B$. We compute via H\"older inequality and coarea formula that 
\begin{eqnarray}\label{eqnarray Y integral bound}
C&\geq & \int_{B}F^{2} \ \omega_{t}^{n} \nonumber
\geq C\int_{0}^{1}\int_{r=\epsilon}F^2 dAd\epsilon,\ \  r=|z|^{\beta}\nonumber
\\&\geq C&\int_{0}^{1}Area^{-1}\{r=\epsilon\}(\int_{r=\epsilon}F dA)^2d\epsilon\nonumber
\\&=& C\int_{0}^{1}\frac{Y}{\epsilon(1-\log \epsilon)}d\epsilon,
\end{eqnarray}
where 
\[Y(\epsilon)=(\int_{r=\epsilon}F dA)^2(1-\log \epsilon).\]

Thus it's easy to deduce the following Claim. 
\begin{clm}\label{clm Y tend to 0 via at least a sequence} There exists a sequence $\epsilon_{k}\rightarrow 0$ such that $Y(\epsilon_{k})\rightarrow 0$.
\end{clm}
If not, then there exist $\delta_{0}$ and $r_{0}$ such that 
$Y(\epsilon)\geq \delta_{0}\ \textrm{for all}\ \epsilon<r_{0}.$

Thus (\ref{eqnarray Y integral bound}) implies 
$c\geq C\delta_{0}\int_{0}^{r_0}\frac{1}{\epsilon(1-\epsilon)}d\epsilon=\infty.$
This is a contradiction. Thus Claim \ref{clm Y tend to 0 via at least a sequence} is true. Hence $\lim_{k\rightarrow \infty}\int_{r=\epsilon_{k}}|\nabla u| dA= 0.$ This implies (\ref{equ lim of v nabla u times n vanishes}) is true via the sequence $\epsilon_{k}$ in Claim \ref{clm Y tend to 0 via at least a sequence}. The proof of (\ref{equ int by parts}) is complete.

   Then the proof of Theorem 10.1 in section 10 of Chapter $III$ 
   in \cite{LSU} directly goes through in $B_p(r_0),\ p\in D$ and $r_0$ is sufficiently small, provided we have $\frac{\partial }{\partial t}dvol_t\leq Cdvol_t.$
   
   \end{proof}

 \begin{lem}\label{Bound on the Scalar Curvature}Along the conical K\"ahler-Ricci flow over $[0,t_0]$ ($t_0$ as in Theorem \ref{Bounding Ricci curvature when t>0}),  the scalar curvature $R$ satisfies $R\geq -\frac{C}{t}$ over $M\setminus D$. In particular, we have 
\begin{equation}\label{Volume form derivative lower bound}
\frac{\partial }{\partial t}dvol_t\leq (\frac{C}{t}+n\beta)dvol_t .
\end{equation}

Moreover, along the perturbed smooth flow (\ref{Perturbed CKRF}), (\ref{Volume form derivative lower bound}) also holds.
\end{lem}
\begin{proof}[ Proof of lemma \ref{Bound on the Scalar Curvature}:] We have \begin{equation}\label{equ super heat equation for scalar curvature}
\frac{\partial (R-n\beta)}{\partial t}\geq \Delta (R-n\beta)+\frac{(R-n\beta)^2}{n}+\beta (R-n\beta).
\end{equation}
Notice   that   Proposition \ref{Bounding Ricci curvature when t>0} implies $R(t)\in C^{\alpha,\beta}$ when $t>0$. It follows from Jeffres' trick as in the proof of Lemma \ref{uniqueness of weak flows} and maximal principles that  $R\geq -\frac{C}{t}$.  To elaborate how to deal with the conical singularity, we show more detail.  By changing notation, the target estimate is \begin{equation}
\tau R(\tau)>-C,\ \tau\in [0,t_{0}].  
\end{equation}
We consider the flow initiated from time $\frac{\tau}{2}$ by letting 
$s=t-\frac{\tau}{2},\ s\in [0, \frac{\tau}{2}]$. It sufficies to show 
\begin{equation}\label{equ scalar s lower bound}
s R(s)\geq -C. 
\end{equation}
The advantage of doing this is that now the regularity of the  metric is improved such that $R(s)\in C^{,\alpha,\frac{\alpha}{2}, \beta}$. Hence $\lim_{s\rightarrow 0} sR(s)=0$ in H\"older continuous sense. Let 
\begin{equation}\label{equ def of udelta}
u_{\delta}=s(R-n\beta)-\delta |S|^{2p},\ p<\frac{\alpha\beta}{2}\ \textrm{as in Claim}\ \ref{clm supremum attained away from D}.
\end{equation}
(\ref{Laplacian of norm of S is bounded from below}) and (\ref{equ super heat equation for scalar curvature}) imply 
\begin{equation}
\frac{\partial u_{\delta}}{\partial s}\geq \Delta u_{\delta}+\frac{u_{\delta}^2}{ns}+\beta u_{\delta}+\frac{u_{\delta}}{s}+\frac{2u_{\delta}\delta |S|^{2p}}{ns}-\delta C.
\end{equation}

Notice that the short existence time $t_{0}$ is very small in the sense of Definition \ref{Convention on the constant}. Let $\delta$ be small enough, we deduce when $u_{\delta}\leq -100n$, the term $\frac{u_{\delta}^2}{ns}$ is much more positive  than the other terms such that 
   \begin{equation}\label{equ derivative of udelta}
\frac{\partial u_{\delta}}{\partial s}\geq \Delta u_{\delta}+1.
\end{equation}

Then 
\begin{clm}\label{clm udelta has lower bound}$ u_{\delta}> -100n $.
\end{clm}
If the claim does not hold, since $u_{\delta}(0)\geq -1$ (when $\delta$ is sufficiently small) and $u_{\delta}$ can not attain spacewise minimum on $D$ (see Claim \ref{clm supremum attained away from D} with sign reversed),  there exists a space-time point $(x_{0},s_{0})$ such that 
\begin{eqnarray*}
& &u_{\delta}(x_{0},s_{0})=-100n,\ u_{\delta}(x,s)\geq -100n\ \textrm{when}\ s<s_{0},
\\& & x_{0},s_{0} \ \textrm{is the spacewise minimum of}\ u_{\delta},\ x_{0}\notin D.
\end{eqnarray*}

Then $\frac{\partial u_{\delta}}{\partial s}(x_{0},s_{0})\leq 0$. But equation (\ref{equ derivative of udelta}) implies $\frac{\partial u_{\delta}}{\partial s}(x_{0},s_{0})\geq 1$, a contradiction. The proof of Claim \ref{clm udelta has lower bound} is complete. Letting $\delta \rightarrow 0$,  Claim \ref{clm udelta has lower bound} implies (\ref{equ scalar s lower bound}). The proof of the conic flow part of   Lemma \ref{Bound on the Scalar Curvature} is complete. The crucial point is that we need the perturbation by $\delta |S|^{2p}$ in (\ref{equ def of udelta}).

 To prove the conclusion for  the perturbed flow, we don't need the perturbation by $\delta |S|^{2p}$ above.  We first denote
 \[v_{\epsilon}=-\frac{\frac{\partial }{\partial t}dvol_t}{dvol_t}.\]
By routine computation we have 

  \begin{equation}
\frac{\partial v_{\epsilon}}{\partial t}\geq \Delta v_{\epsilon}+\frac{v_{\epsilon}^2}{n}+\beta v_{\epsilon},
\end{equation}
where $v_{\epsilon}$ is a smooth function on which the maximal principle directly works. Then (\ref{Volume form derivative lower bound}) is true for the perturbed flow (\ref{Perturbed CKRF}).
\end{proof}

\section{Bootstrapping of the  conical K\"ahler-Ricci flow and proofs of the main results in the introduction.}
In this short section, we show the solutions to the conical K\"ahler-Ricci flow 
possess maximal regularity when $t>0$ by proving Theorem (\ref{Bounding Ricci curvature when t>0}). This shows the requirements in Theorem \ref{Existence in the weak flow0} are satisfied when $t=\frac{t_0}{2}$ ($t_0$ as in Theorem \ref{Bounding Ricci curvature when t>0}). 

\begin{proof}[ Proof of Theorem \ref{Bounding Ricci curvature when t>0}:]% Step 1. From \cite{WYQ}, $\frac{\partial \phi}{\partial t}$ has maximal regularity and  the interior estimates hold if for any $\delta>0$, we have   $\frac{\partial \phi}{\partial t}\in W^{1,2}(M\times [\delta,t_0])$ and $\frac{\partial^2 \phi}{\partial t^2}\in  L^{2}(M\times [\delta,t_0])$ (so that 
%$\frac{\partial \phi}{\partial t}$ is a weak solution to a parabolic equation). To be precise, assuming Claim \ref{W12 estimate}, then the conclusion of Proposition \ref{Bounding Ricci curvature when t>0} is straight forward.
   \ Denote $v=e^{-\beta t}\frac{\partial \phi}{\partial t}$, then $v$ satisfies the following equation 
\begin{equation}\label{parabolic equation for the twisted potential}\frac{\partial v}{\partial t}=\Delta_{\phi}v.
\end{equation} 
Suppose   $v\in C^{2+\alpha,1+\frac{\alpha}{2},\beta}(0,T]$ (Claim \ref{clm regularity of v}), using the parabolic interior Schauder estimate as in the equation (21) in \cite{CYW}, we obtain the following 
\begin{equation}\label{Bounding the Ricci potential along the flow}
|i\partial \bar{\partial} v|^{(2)}_{\alpha,\frac{\alpha}{2},\beta,M\times [0,t_0]}+|\frac{\partial v}{\partial t}|^{(2)}_{\alpha,\frac{\alpha}{2},\beta,M\times [0,t_0]}\leq C|v|_{0,M\times [0,t_0]}.
\end{equation}
Then equation (\ref{Bounding the Ricci potential along the flow}) directly  implies Theorem \ref{Bounding Ricci curvature when t>0}, because $i\partial \bar{\partial}\frac{\partial \phi}{\partial t}=-Ric+\beta\omega$ over $M\setminus D$.

  Thus it suffices to show 
  \begin{clm}\label{clm regularity of v} $v\in C^{2+\alpha,1+\frac{\alpha}{2},\beta}(0,T]$.
  \end{clm}
  The claim is proved in a very easy way as follows. For any $\delta>0$, choose a timewise-cutoff function $\eta(t)$ such that 
  \[\eta(t)=0\ \textrm{when}\ t\leq \delta,\ \eta(t)=1\ \textrm{when}\ t\geq 2\delta.\]  
  It suffices to prove  $\eta v\in C^{2+\alpha,1+\frac{\alpha}{2},\beta}[0,T]$.  First we compute 
  \begin{equation}\label{equ evolution of eta v}
  \frac{\partial (\eta v)}{\partial t}=\Delta_{\phi}(\eta v)+\eta^{\prime}v.
  \end{equation}
  Since $ v\in C^{\alpha,\frac{\alpha}{2},\beta}[0,T]$, by Theorem 1.8 in \cite{CYW}, there exists a solution $U\in C^{2+\alpha,1+\frac{\alpha}{2},\beta}[0,T]$ which solves
  \begin{equation}\label{equ evolution of eta v}
  \frac{\partial  U}{\partial t}=\Delta_{\phi} U+\eta^{\prime}v,\ U(x,0)=0.
  \end{equation}
  Consider $W=U-\eta v$, then 
  \begin{equation}\label{equ regularity of W}
  W\in C^{\alpha,\frac{\alpha}{2},\beta}[0,T]
  \end{equation}  and 
   \begin{equation}\label{equ evolution of W}
  \frac{\partial  W}{\partial t}=\Delta_{\phi}W,\ W(x,0)=0 \ \textrm{over}\ M\setminus D.
  \end{equation}
 Using  Jeffres's trick as in the proof of Lemma (\ref{uniqueness of weak flows}), (\ref{equ regularity of W}) and (\ref{equ evolution of W}) imply $W=0$. Then $\eta v=U\in C^{2+\alpha,1+\frac{\alpha}{2},\beta}[0,T]$.
  \end{proof}

%Step 2: Now the only task is to prove the following claim.
% \begin{clm}\label{W12 estimate}$v\in W^{1,2}(M\times [\delta,t_0])$
% and $\frac{\partial v}{\partial t}\in  L^{2}(M\times [\delta,t_0])$.
%\end{clm}
%The proof of the claim is not hard. 
%For any $p\notin D$, suppose $r_p=dist_{\omega_{D}}(p,D)$.  We can apply the usual interior gradient estimates for $v$ to equation (\ref{parabolic equation for the twisted potential}) in  in $B_{p}(\frac{r_p}{2})$ to get 
%\begin{eqnarray*}& &
%|\nabla_{\phi} v|_{0,B_{p}(\frac{r_p}{3})\times [\frac{t}{3},t]}
%\\&\leq & \frac{C}{\min\{r^{1-\alpha}_{p},t^{\frac{1-\alpha}{2}}\}}%[v]_{\alpha,\frac{\alpha}{2},\beta, B_{p}(\frac{r_p}{2})\times [\frac{t}{2},t]}
%\\&\leq &  \frac{C}{r^{1-\alpha}_{p}t^{\frac{1-\alpha}{2}}}.
%\end{eqnarray*}
%Thus the blown-up rate of $|\nabla \frac{\partial \phi}{\partial t}|$ is as $\frac{1}{r^{1-\alpha}}$ when $t>0$, which is obviously square integrable. To be precise, we obviously have 
%\begin{equation*}\int_{M}|v|^2d\omega_{\phi}^n+\int^{t_0}_{0}\int_{M}|%\nabla_{\phi} v|^2d\omega_{\phi}^ndt\leq C<+\infty.
%\end{equation*}
%The proof of Claim \ref{W12 estimate} is complete.

\begin{proof}[Proof of Theorem \ref{long time existence of the weak flow} and Theorem \ref{Thm Smooth approximation of the conical flow}:]

When $t=\frac{t_0}{2}$, we define 
\begin{equation}F_{\frac{t_0}{2}}\triangleq\log \frac{|S|^{2-2\beta}\omega^n_{\phi_0}}{\omega^n_{0}}=-h+\frac{\partial \phi}{\partial t}|_{\frac{t_0}{2}}-\beta \phi.
\end{equation}
By Theorem \ref{Bounding Ricci curvature when t>0}, we directly get
\begin{equation}
F_{\frac{t_0}{2}}\in C^{2,\alpha,\beta},\ |F_{\frac{t_0}{2}}|_{2,\alpha,\beta}\leq C.
\end{equation}
Theorem \ref{Existence in the weak flow0} implies the long time existence of the weak flow for all time $t\in [\frac{t_0}{2}, \infty)$. 
Combining  the strong conical flow over $[0,\frac{t_0}{2}]$,  we end up with a strong flow for all $t\in [0,\infty)$, for any $C^{2,\alpha,\beta}$ initial potential $\phi_{0}$.

 Lemma \ref{uniqueness of weak flows} implies the weak flow  coincides with the  strong flow given by Theorem 1.2 in \cite{CYW} till whenever the strong flow exists.
 
By approximating the conical flow over  $[\frac{a_{0}}{2}, T)$ by the flows (\ref{Perturbed CKRF}) as in  section \ref{Construction of the approximation flows},  Theorem \ref{Thm Smooth approximation of the conical flow} follows from
 the second-order estimates in  Proposition \ref{Parabolic Laplacian estimate before rescaling}, the parabolic Evans-Krylov estimate over $M\setminus T_{\delta}(D)$ (as in \cite{Wanglihe1}, $T_{\delta}(D)$ defined in Remark \ref{rmk turbular neighborhood}), and the arguments in Proposition 2.5 in \cite{CDS1}.
\end{proof}

Yuanqi Wang, Department of Mathematics, University of California  at Santa Barbara, Santa Barbara,
CA,  USA.
 
\end{document}